\documentclass[11pt]{amsart}
\usepackage[english]{babel}
\usepackage{graphicx}
\newtheorem{theorem}{Theorem}[section]

\newtheorem{lemma}[theorem]{Lemma}
\newtheorem{proposition}[theorem]{Proposition}
\theoremstyle{definition}

\theoremstyle{remark}

\numberwithin{equation}{section}
\newcommand{\real}{\mathbb R}

\def\natu{\mathbb N}
\def\dis{\displaystyle}

\begin{document}

\title[Lyapunov inequalities for PDE]
{Lyapunov inequalities for Partial Differential Equations at
radial higher eigenvalues}
\author{Antonio Ca\~{n}ada}
\author{Salvador Villegas}
\thanks{The authors have been supported by the Ministry of Education and
Science of Spain (MTM2008.00988)}
\address{Departamento de An\'{a}lisis
Matem\'{a}tico, Universidad de Granada, 18071 Granada, Spain.}
\email{acanada@ugr.es, svillega@ugr.es}


\dedicatory{To our dear friend and colleague Jean Mawhin on the occasion of his seventieth birthday}%
\begin{abstract}
This paper is devoted to the study of $L_{p}$ Lyapunov-type
inequalities ($ \ 1 \leq p \leq +\infty$) for linear partial
differential equations at radial higher eigenvalues. More
precisely, we treat the case of Neumann boundary conditions on
balls in $\real^{N}$.  It is proved that the relation between the
quantities $p$ and $N/2$ plays a crucial role to obtain nontrivial
and optimal Lyapunov inequalities. By using appropriate minimizing
sequences and a detailed analysis about the number and
distribution of zeros of radial nontrivial solutions, we show
significant qualitative differences according to the studied case
is subcritical, supercritical or critical.
\end{abstract}

\maketitle

\section{Introduction}

\noindent Let us consider the linear problem
\begin{equation}\label{ceq2}
u''(x) + a(x)u(x) = 0, \ x \in (0,L), \ u'(0) = u'(L) = 0
\end{equation}
where $a \in \Lambda$ and $\Lambda$ is defined by
\begin{equation}\label{ceq3}
\Lambda = \{ a \in L^1(0,L)\setminus \{0\}: \dis \int_{0}^{L} a(x)
\ dx \geq 0 \ \ \mbox{and} \ (\ref{ceq2}) \ \mbox{has nontrivial
solutions} \ \}
\end{equation}
The well known $L_1$ Lyapunov inequality states that if $a \in
\Lambda,$ then $\displaystyle \int_{0}^{L} a^{+}(x) \ dx
> 4/L.$ Moreover, the constant $4/L$ is optimal since $\displaystyle \frac{4}{L} = \displaystyle
\inf_{a \in \Lambda} \Vert a^+ \Vert_{L^1(0,L)}$ and this infimum
is not attained (see \cite{camovimia}, \cite{hart} and
\cite{huyo}). This result is as a particular case of the so called
$L_p$ Lyapunov inequalities, $1 \leq p \leq \infty.$ In fact, if
for each $p$ with $1 \leq p \leq \infty,$ we define the quantity
\begin{equation}\label{ceq4}
\beta_{p} \equiv \inf_{a \in \Lambda\cap L^p (0,L)} \ I_{p}(a)
\end{equation}
where
\begin{equation}\label{ceq5}
\begin{array}{c}
I_{p}(a) = \Vert a^{+} \Vert_{L^p (0,L)} = \left( \displaystyle
\int_{0}^{L} (a^{+}(x)) ^{p} \ dx \right) ^{1/p}, \forall
\ a \in \Lambda\cap L^p (0,L), \ 1 \leq p < \infty, \\ \\
I_{\infty}(a) =\mbox{sup ess} \ a^{+}, \ \forall \ a \in
\Lambda\cap L^\infty (0,L),
\end{array}
\end{equation}
\noindent then $\beta_1 = \frac{4}{L}$ and for each $p$ with $1
\leq p \leq \infty,$ it is possible to obtain an explicit
expression for $\beta_{p}$ as a function of $p$ and $L$
(\cite{camovimia}, \cite{zhan}).

Let us observe that the real number zero is the first eigenvalue
of the eigenvalue problem
\begin{equation}\label{n2}
u''(x) + \rho u(x) = 0, \ x \in (0,L), \ u'(0) = u'(L) = 0
\end{equation}
and that for Neumann boundary conditions the restriction on the
function $a$ in the definition of the set $\Lambda,$
\begin{equation}\label{2106111}
a \in L^1(0,L)\setminus \{0\}, \ \dis \int_{0}^{L} a(x) \ dx \geq
0,
\end{equation}
or the more restrictive pointwise condition
\begin{equation}\label{2106112}
a \in L^1 (0,L), \ 0 \prec a,
\end{equation}
are natural if we want to obtain nontrivial optimal Lyapunov
inequalities (see Remark 4 in \cite{camovimia}). Here, for $c,d
\in L^1 (0,L),$ we write $c \prec d$ if $c(x) \leq d(x)$ for a.e.
$x \in [0,L]$ and $c(x) < d(x)$ on a set of positive measure.

In fact, it can be easily proved that if
\begin{equation}\label{bisceq3}
\Lambda_{0} = \{ a \in L^1(0,L): 0 \prec a \ \ \mbox{and} \
(\ref{ceq2}) \ \mbox{has nontrivial solutions} \ \}
\end{equation}
then the constant $\beta_p$ defined in (\ref{ceq4}) satisfies
\begin{equation}\label{bisceq4}
\beta_{p} =  \inf_{a \in \Lambda_{0}\cap L^p (0,L)} \ I_{p}(a)
\end{equation}
Since zero is the first eigenvalue of (\ref{n2}), it is coherent
to affirm that $\beta_p$ is the $L_p$ Lyapunov constant for the
Neumann problem at the first eigenvalue.

On the other hand, the set of eigenvalues of (\ref{n2}) is given
by $\rho_k = k^2 \pi^2/L^2, \ k \in \natu \cup \{ 0 \}$ and if for
each $k \in \natu \cup \{ 0 \},$ we consider the set

\begin{equation}\label{landaenedef}
\Lambda_{k} = \{ a \in L^1(0,L): \ \rho_k \prec a  \ \ \mbox{and}
\ (\ref{ceq2}) \ \mbox{has nontrivial solutions} \ \}
\end{equation} then for each $p$ with $1 \leq p \leq \infty,$
 we can define the constant
\begin{equation}\label{nceq4}
\beta_{p,k} \equiv \inf_{a \in \Lambda_{k}\cap L^p (0,L)} \
I_{p}(a-\rho_k)
\end{equation}
An explicit value for $\beta_{1,k}$ has been obtained by the
authors in \cite{cavijems}. The case $p = \infty$ is trivial
($\beta_{\infty,k} = \rho_{k+1} - \rho_k $) and, to the best of
our knowledge, an explicit value of $\beta_{p,k}$ as a function of
$p,k$ and $L$ is not known when $1 < p < \infty$. Nevertheless,
since $\beta_{1,k} >0,$ we trivially deduce $\beta_{p,k} >0,$ for
each $p$ with $1 \leq p \leq \infty.$

$ $

With regard to Partial Differential Equations, the linear problem
\begin{equation}\label{ceq6}
\left.
\begin{array}{cl}
\Delta u(x)+a(x)u(x)= 0, & x\in \Omega \, \\
 \frac{\partial u}{\partial n}(x)=0,\ \ \ \ \ \ \ &  x\in \partial \Omega
\end{array}
\right\}
\end{equation}
\noindent has been studied in \cite{camovijfa}, where $\Omega
\subset \real^N$ ($N\geq 2$) is a bounded and regular domain,
$\dis \frac{\partial}{\partial n}$ is the outer normal derivative
on $\partial \Omega$ and the function $a:\Omega \rightarrow \real$
belongs to the set $\Gamma$ defined as
\begin{equation}\label{ceq7}
\Gamma = \{ a \in L^{\frac{N}{2}}(\Omega)\setminus \{0\}: \dis
\int_\Omega a(x) \ dx \geq 0 \ \mbox{and} \ (\ref{ceq6}) \
\mbox{has nontrivial solutions}\}
\end{equation}
\noindent if $N \geq 3$ and
$$\Gamma = \{ a:\Omega \rightarrow \real \mbox{ s. t. } \exists
q\in(1,\infty]\\ \ \mbox{with} \ a\in L^q(\Omega)\setminus \{0\},
\
 \dis \int_\Omega a(x) \ dx \geq 0 \ $$

\noindent\  and \ (\ref{ceq6}) has nontrivial solutions$\}$
\newline
\noindent if $N = 2.$ \newline\ Obviously, the quantity
\begin{equation}\label{ceq8}
\gamma_{p} \equiv \inf_{a \in \Gamma\cap L^p (\Omega)} \ \Vert a
^{+} \Vert_{L^p (\Omega)} \ , \ 1 \leq p \leq \infty
\end{equation}
is well defined and it is a nonnegative real number. A remarkable
novelty (see \cite{camovijfa}) with respect to the ordinary case
is that $\gamma_{1} = 0$ for each $N \geq 2.$ Moreover, if $N =2,$
then $\gamma_{p} > 0, \ \forall \ p \in (1,\infty]$ and if $N \geq
3,$ then $\gamma_{p}
> 0$ if and only if $p \geq N/2.$ In contrast to the ordinary case, it seems difficult to obtain an explicit
expressions for $\gamma_{p},$ as a function of $p$, $\Omega$ and $N,$
at least for general domains.

As in the ordinary case, the real number zero is the first
eigenvalue of the eigenvalue problem
\begin{equation}\label{nceq11}
\left.
\begin{array}{cl}
\Delta u(x)+ \rho u(x)= 0,  & x\in \Omega \, \\
 \frac{\partial u}{\partial n}(x)=0\, \ \ \ \ \ \ &  x\in \partial \Omega
\end{array}
\right\}
\end{equation}
so that it is natural to say that the constant $\gamma_p$ defined
in (\ref{ceq8}) is the $L_p$ Lyapunov constant at the first
eigenvalue for the Neumann problem (\ref{ceq6}).

$ $

To our knowledge, there are no significant results concerning to
$L_p$ Lyapunov inequalities for PDE at higher eigenvalues and this
is the main subject of this paper where we provide some new
qualitative results which extend to higher eigenvalues those
obtained in \cite{camovijfa} for the case of the first eigenvalue.
We carry out a complete qualitative study of the question pointing
out the important role played by the dimension of the problem.

Since in the case of ODE our proof are mainly based on an exact
knowledge about the number and distribution of the zeros of the
corresponding solutions (\cite{cavijems}), in the PDE case we are
able to study $L_p$
 Lyapunov inequalities if $\Omega$ is a ball and for radial higher
eigenvalues. It is not restrictive to assume that $\Omega =
B_{\real^N} (0;1) \equiv B_1,$ the open ball in $\real^N$ of
center zero and radius one.

In Section 2 we describe the problem in a precise way and we
present the main results of this paper. In Section 3 we study the
subcritical case, i.e. $1\leq p<\frac{N}{2}$, if $N\geq 3$, and
$p=1$ if $N=2$. To prove the results in this section we will
construct some explicit and appropriate sequences of problems like
(\ref{ceq6}) where Dirichlet type problems play an essential role.
In this subcritical case we prove that the optimal Lyapunov
constants are trivial, i.e., zero.

In Section 4, we treat with the supercritical case:
$p>\frac{N}{2}$, if $N\geq 2$. By using some previous results of
Section 2, about the number and distribution of the zeros of
nontrivial and radial solutions, together with some compact
Sobolev inclusions, we use a reasoning by contradiction to prove
that the optimal Lyapunov constants are strictly positive and they
are attained. In Section 5 we consider the critical case, i.e.
$p=\frac{N}{2}$, if $N\geq 3$. Because in this case the Sobolev
inclusions are continuous but no compact, we demonstrate that the
optimal Lyapunov constants are strictly positive but we do not
know if they are attained or not.

Finally,  we study the case of Neumann boundary conditions but
similar results can be obtained in the case of Dirichlet type
problems.

\section{main results}
From now on, $\Omega = B_1,$ the open ball in $\real^N$ of center
zero and radius one. It is very well known (\cite{delpino}) that
the operator $-\Delta$ exhibits an infinite increasing sequence of
radial Neumann eigenvalues $ 0 = \mu_0 < \mu_1<\ldots < \mu_k <
\ldots $ with $\mu_k \rightarrow + \infty,$ all of them simple and
with associated eigenfunctions $\varphi_k \in C^1 [0,1]$ solving
\begin{equation}\label{2106114}
\begin{array}{c}
-(r^{N-1} \varphi ')' = \mu_k r^{N-1} \varphi, \ 0 < r < 1, \\
\varphi'(0) = \varphi'(1) = 0.
\end{array}
\end{equation}
Moreover, each eigenfunction $\varphi_k$ has exactly $k$ simple
zeros $r_k<r_{k-1}<...<r_1$ in the interval $(0,1)$.

For each integer $k\geq 0$ and number $p$, $1 \leq p \leq \infty,$
we can define the set $$ \Gamma_k = \{ a \in L^{N/2}(B_1): \ a \
\mbox{is a radial function}, \ \mu_k \prec a \ \mbox{and} $$
(\ref{ceq6}) \ \mbox{has radial and nontrivial solutions} $\}$
\noindent if $N\geq 3$ and
$$\Gamma_k = \{ a:B_1 \rightarrow \real \mbox{ s. t. } \exists
q\in(1,\infty]\\ \ \mbox{with} \ a\in L^q(B_1): \ a \ \mbox{is
a radial function},$$
$$\mu_k \prec a \ \mbox{and} \ (\ref{ceq6}) \ \mbox{has radial and
nontrivial solutions} \}$$

\noindent if $N=2$.

We also define the quantity

\begin{equation}\label{2106116}
\gamma_{p,k} = \displaystyle \inf_{a \in \Gamma_k \cap L^p (B_1)}
\ \Vert a-\mu_k \Vert _{L^p (B_1)}
\end{equation}

The main result of this paper is the following.

\begin{theorem}\label{principal}

 Let $k\geq 0$, $N\geq 2$, $1\leq p\leq \infty$. The following statements hold:
\begin{enumerate}
\item If $N=2$ then $\gamma_{p,k}>0 \Leftrightarrow 1<p\leq \infty
$.

\noindent If $N\geq 3$ then $\gamma_{p,k}>0 \Leftrightarrow
\frac{N}{2}\leq p\leq \infty $.

\

\item If $N \geq 2$ and $\frac{N}{2}<p\leq \infty$ then
$\gamma_{p,k}$ is attained.

\end{enumerate}

\end{theorem}

A key ingredient to prove this theorem is the following
proposition on the number and distribution of zeros of nontrivial
radial solutions of (\ref{ceq6}) when $a \in \Gamma_k.$

\begin{proposition}\label{n+1ceros}
Let $\Omega=B_1$, $k\geq 0$, $a \in \Gamma_k$ and $u$ any
nontrivial radial solution of (\ref{ceq6}). Then $u$ has, at
least, $k+1$ zeros in $(0,1)$. Moreover, if $k\geq 1$ and we
denote by $x_k<x_{k-1}<...<x_1$ the last $k$ zeros of $u$, we have
that

$$r_i\leq x_i\, ,\mbox{   }\forall \ 1\leq i\leq k,$$

\noindent where $r_i$ denotes de zeros of the eigenfunction
$\varphi_k$ of (\ref{2106114}).
\end{proposition}

For the proof of this proposition we will need the following
lemma. Some of the results of this lemma can be proved in a
different way, by using the version of the Sturm Comparison Lemma
proved in \cite{delpino}, Lemma 4.1, for the $p$-laplacian
operator (see also \cite{hart}). Other results are new.

\begin{lemma}\label{estimas}
Let $k\geq 1$. Under the hypothesis of Proposition \ref{n+1ceros}
we have that

\begin{enumerate}
\item[i)] $u$ vanishes in the interval $(0,r_k]$. If $r_k$ is the
only zero of $u$ in this interval then $a(r)\equiv \mu_k$ in
$(0,r_k]$. \item[ii)] $u$ vanishes in the interval
$[r_{i+1},r_i)$, for $1\leq i\leq k-1$. If $r_{i+1}$ is the only
zero of $u$ in this interval then $u(r_i)=0$ and $a(r)\equiv
\mu_k$ in $[r_{i+1},r_i]$. \item[ii)]  $u$ vanishes in the
interval $[r_1,1)$. If $r_1$ is the only zero of $u$ in this
interval then $a(r)\equiv \mu_k$ in $[r_1,1]$.
\end{enumerate}
\end{lemma}

\begin{proof}

To prove i), multiplying (\ref{ceq6}) by $\varphi_k$ and
integrating by parts in $B_{r_k}$ (the ball centered in the origin
of radius $r_k$), we obtain

$$\int_{B_{r_k}}\nabla u\nabla \varphi_k=\int_{B_{r_k}}a u \varphi_k.$$

On the other hand, multiplying (\ref{2106114}) by $u$ and
integrating by parts in $B_{r_k}$, we have

$$\int_{B_{r_k}}\nabla \varphi_k\nabla u=\mu_k\int_{B_{r_k}}\varphi_k u+\int_{\partial B_{r_k}}u\frac{\partial \varphi_k}{\partial n}.$$

Subtracting these equalities yields

\begin{equation}\label{signo}
\int_{B_{r_k}}\left(a-\mu_k\right)u\varphi_k=\omega_N
r_k^{N-1}u(r_k)\varphi_k'(r_k),
\end{equation}

\noindent where $\omega_N$ denotes de measure of the
$N$-dimensional unit sphere. Assume, by contradiction, that $u$
does not vanish in $(0,r_k]$. We can suppose, without loss of
generality, that $u>0$ in this interval. We can also assume that
$\varphi_k>0$ in $(0,r_k)$. Since $r_k$ is a simple zero of
$\varphi_k$, we have $\varphi_k'(r_k)<0$ and since $a\geq\mu_k$ in
$(0,r_k)$ we obtain a contradiction.

Finally, if $r_k$ is the only zero of $u$ in $(0,r_k]$, equation
\ref{signo} yields \newline
$\int_{B_{r_k}}\left(a-\mu_k\right)u\varphi_k=0$, which gives
$a(r)\equiv \mu_k$ in $(0,r_k]$.

To deduce  ii), we proceed similarly to the proof of part i),
restituting $B_{r_k}$ by $A(r_{i+1},r_i)$ (the annulus centered in
the origin of radii $r_{i+1}$ and $r_i$) and obtaining

$$\int_{A(r_{i+1},r_i)}\left(a-\mu_k\right)u\varphi_k=\omega_N r_i^{N-1}u(r_i)\varphi_k'(r_i)-\omega_N r_{i+1}^{N-1}u(r_{i+1})\varphi_k'(r_{i+1})$$

\noindent and ii) follows easily by arguments on the sign of these
quantities, as in the proof of part i).

To obtain iii), a similar analysis to that in the previous cases
shows that

$$\int_{A(r_i,1)}\left(a-\mu_k\right)u\varphi_k=-\omega_N r_1^{N-1}u(r_1)\varphi_k'(r_1),$$

\noindent and the lemma follows easily as previously.

\end{proof}

{\it Proof of Proposition \ref{n+1ceros}.} Let $k=0$. If we
suppose that $u$ has no zeros in $(0,1]$ and we integrate the
equation $-\Delta u=a\, u$ in $B_1$, we obtain $\int_{B_1}a\,
u=0$, a contradiction. Hence, for the rest of the proof we will
consider $k\geq 1$.

Let $1\leq i\leq k$. By the previous lemma $u$ vanishes in the $i$
disjoint intervals $[r_i,r_{i-1})$,...,$[r_2,r_1)$,$[r_1,1)$.
Therefore $u$ has, at least, $i$ zeros in the interval $[r_i,1)$
which implies that $r_i\leq x_i$.

Finally, let us prove that $u$ has, at least, $k+1$ zeros. From
the previous part, taking $i=k$, $u$ has at least $k$ zeros in the
interval $[r_k,1]$, one in each of the $k$ disjoint intervals
$[r_k,r_{k-1})$,...,$[r_2,r_1)$,$[r_1,1)$. Suppose, by
contradiction, that these are the only zeros of $u$. Then $u$ does
not vanish in $(0,r_k)$ and applying part i) of Lemma
\ref{estimas} we obtain $u(r_k)=0$ and $a\equiv\mu_k$ in
$(0,r_k].$ Applying now part ii) of this lemma, we deduce
$u(r_{k-1})=0$ and $a\equiv\mu_k$ in $[r_k,r_{k-1}]$. Repeating
this argument and using part iii) of the previous lemma we
conclude $u(r_i)=0$, for all $1\leq i\leq k$ and $a\equiv\mu_k$ in
(0,1], which contradicts $a\in\Gamma_k$. \qed

$ $

\noindent For the proof of Theorem \ref{principal}, we will
distinguish three cases: the subcritical case ($1\leq
p<\frac{N}{2}$ if $N\geq 3$, and $p=1$ if $N=2$), the
supercritical case ($p>\frac{N}{2}$ if $N\geq 2$), and the
critical case ($p=\frac{N}{2}$ if $N\geq 3$).

\section{The subcritical case}

In this section, we study the subcritical case, i.e. $1\leq
p<\frac{N}{2}$, if $N\geq 3$, and $p=1$ if $N=2$. In all those
cases we will prove that $\gamma_{p,k}=0$.

$ $

The next lemma is related to the continuous domain dependence of
the eigenvalues of the Dirichlet Laplacian. In fact, the result is
valid under much more general hypothesis (see \cite{fugle}). Here
we show a very simple proof for this special case.

\begin{lemma}\label{autovalorescontinuos}

Let $N\geq 2$ and $R>0$. Then

$$\lim_{\varepsilon \to 0}\lambda_1\left( A(\varepsilon,R)\right)=\lambda_1\left(B_R\right),$$

\noindent where $\lambda_1\left( A(\varepsilon,R)\right)$ and
$\lambda_1\left(B_R\right)$ denotes, respectively, the first
eigenvalues of the Laplacian operator with Dirichlet boundary
conditions of the annulus $A(\varepsilon,R)$ and the ball $B_R$.

\end{lemma}

\begin{proof}

For $N\geq 3$ and $\varepsilon \in (0,R/2)$ define the following
radial function $u_\varepsilon\in
H_0^1\left(A(\varepsilon,R)\right)$:

\begin{equation} u_\varepsilon(x) = \left \{%
\begin{array}{l}
\phi_1(x), \ \ \ \ \ \ \ \ \ \ \ \ \ \ \mbox{if} \ \ \ 2\varepsilon\leq \vert x\vert < R , \\ \\
\displaystyle{\frac{\vert
x\vert-\varepsilon}{\varepsilon}}\phi_1(2\varepsilon), \
\mbox{if} \ \ \ \varepsilon <\vert x\vert < 2\varepsilon ,
\end{array}
\right.
\end{equation}

\noindent where $\phi_1$ denotes the first eigenfunction with
Dirichlet boundary conditions of the ball $B_R$. It is easy to
check that
$$\displaystyle{\lim_{\varepsilon\to
0}\int_{A(\varepsilon, 2\varepsilon)}\vert\nabla
u_\varepsilon\vert^2}=\lim_{\varepsilon\to
0}\int_\varepsilon^{2\varepsilon}\omega_Nr^{N-1}\frac{\phi_1(2\varepsilon)^2}{\varepsilon^2}dr=0.$$
In the same way it is obtained $\displaystyle{\lim_{\varepsilon\to
0}\int_{A(\varepsilon, 2\varepsilon)}u_\varepsilon^2=0}$. In
addition, from the variational characterization of the first
eigenvalue it follows that $\lambda_1\left(
A(\varepsilon,R)\right)\leq \int_{A(\varepsilon,R)}\vert\nabla
u_\varepsilon\vert^2/\int_{A(\varepsilon,R)}u_\varepsilon^2$.
Therefore

$$\limsup_{\varepsilon\to 0}\lambda_1\left( A(\varepsilon,R)\right)\leq \limsup_{\varepsilon\to 0}\displaystyle{\frac{\int_{A(\varepsilon,R)}\vert\nabla u_\varepsilon\vert^2}{\int_{A(\varepsilon,R)}u_\varepsilon^2}=\frac{\int_{B_R}\vert\nabla \phi_1\vert^2}{\int_{B_R}\phi_1^2}}=\lambda_1(B_R).$$

On the other hand, using that the first Dirichlet eigenvalue
$\lambda_1(\Omega)$ is strictly decreasing with respect to the the
domain $\Omega$, it follows that $\lambda_1\left(
A(\varepsilon,R)\right)>\lambda_1\left(B_R\right)$. Thus

$$\liminf_{\varepsilon\to 0}\lambda_1\left( A(\varepsilon,R)\right)\geq\lambda_1(B_R)$$ and the lemma follows for $N\geq 3$.

The same proof works for $N=2$ if we consider, for every
$\varepsilon\in \left(0,\min\{ 1,R^2\}\right)$, the radial
function $u_\varepsilon\in H_0^1\left(A(\varepsilon,R)\right)$:

\begin{equation} u_\varepsilon(x) = \left \{%
\begin{array}{l}
\phi_1(x), \ \ \ \ \ \ \ \ \ \ \ \ \ \ \mbox{if} \ \ \ \sqrt{\varepsilon}\leq \vert x\vert < R , \\ \\
\displaystyle{\frac{\log\vert
x\vert-\log\varepsilon}{\log\sqrt{\varepsilon}-\log\varepsilon}\phi_1(\sqrt{\varepsilon)}},
\  \mbox{if} \ \ \ \varepsilon <\vert x\vert < \sqrt{\varepsilon}.

\end{array}
\right.
\end{equation}

\end{proof}

\begin{lemma}\label{subcritico}

Let $k\geq 0$, $N\geq 3$ and $1\leq p< N/2$. Then
$\gamma_{p,k}=0$.

\end{lemma}

\begin{proof} If $k=0$, this lemma follows from \cite[Lem. 3.1]{camovijfa}. In this lemma a family of bounded, positive and radial solutions were used. Hence, for the rest of the proof we will consider $k\geq 1$.

To prove this lemma we will construct an explicit family
$a_\varepsilon\in \Gamma_k$ such that $\lim_{\varepsilon\to
0}\Vert a_\varepsilon-\mu_k \Vert_{L^p (B_1)} = 0.$ To this end,
for every $\varepsilon\in(0,r_k)$, define
$u_\varepsilon:B_1\rightarrow \real$ as the radial function

\begin{equation}\label{solucionsubcritico} u_\varepsilon = \left \{%
\begin{array}{l}
\varphi_k, \ \ \ \ \ \ \ \ \ \ \ \ \ \ \mbox{if} \ \ \ r_k\leq \vert x\vert < 1 , \\ \\
\phi_1\left( A(\varepsilon,r_k)\right), \  \mbox{if} \ \ \ \varepsilon \leq \vert x\vert < r_k , \\ \\
\phi_1\left(B_\varepsilon\right), \ \ \ \ \ \ \ \ \mbox{if} \ \ \
\vert x\vert < \varepsilon .
\end{array}
\right.
\end{equation}

\noindent where $\phi_1\left( A(\varepsilon,r_k)\right)$ and
$\phi_1\left(B_\varepsilon\right)$ denotes, respectively, the
first eigenfunctions with Dirichlet boundary conditions of the
annulus $A(\varepsilon,r_k)$ and the ball $B_\varepsilon$.
Moreover these eigenfunctions are chosen such that $u_\varepsilon
\in C^1(\overline{B_1})$.

Then, it is easy to check that $u_\varepsilon$ is a solution of
(\ref{ceq6}), being $a_\varepsilon\in L^\infty(B_1)$ the radial
function

\begin{equation}\label{infimosubcritico} a_\varepsilon = \left \{%
\begin{array}{l}
\mu_k, \ \ \ \ \ \ \ \ \ \ \ \ \ \ \mbox{if} \ \ \ r_k< \vert x\vert < 1 , \\ \\
\lambda_1\left( A(\varepsilon,r_k)\right), \  \mbox{if} \ \ \ \varepsilon < \vert x\vert < r_k , \\ \\
\lambda_1\left(B_\varepsilon\right), \ \ \ \ \ \ \ \ \mbox{if} \ \
\ \vert x\vert < \varepsilon ,
\end{array}
\right.
\end{equation}

\noindent where $\lambda_1\left( A(\varepsilon,r_k)\right)$ and
$\lambda_1\left(B_\varepsilon\right)$ denotes, respectively, the
first eigenvalues with Dirichlet boundary conditions of de annulus
$A(\varepsilon,r_k)$ and the ball $B_\varepsilon$. Since the first
Dirichlet eigenvalue $\lambda_1(\Omega)$ is strictly decreasing
with respect to the the domain $\Omega$, it follows that

$$\lambda_1\left( A(\varepsilon,r_k)\right), \lambda_1\left(B_\varepsilon\right)>\lambda_1\left(B_{r_k}\right)=\mu_k,$$

\noindent which gives $a_\varepsilon\in \Gamma_k$. (The equality $\lambda_1\left(B_{r_k}\right)=\mu_k$ follows from the fact that $\varphi_k$ is a positive solution of $-\Delta \varphi=\mu_k \varphi$ in $B_{r_k}$ which vanishes on $\partial B_{r_k}$). Let us estimate
the $L_p$-norm of $a_\varepsilon-\mu_k$:

\begin{equation}\label{chorizo} \begin{array}{l}
\dis \Vert a_\varepsilon-\mu_k
\Vert_{L^p(B_1)}=\left(\int_{B_\varepsilon}\left(\lambda_1\left(B_\varepsilon\right)-\mu_k\right)^p+\int_{A(\varepsilon,r_k)}\left(\lambda_1\left(A(\varepsilon,r_k)\right)-\mu_k\right)^p
\right)^{\frac{1}{p}}= \\
\dis
=\left(\left(\lambda_1\left(B_\varepsilon\right)-\mu_k\right)^p\frac{\omega_N
\varepsilon^N}{N}+\left(\lambda_1\left(A(\varepsilon,r_k)\right)-\mu_k\right)^p\frac{\omega_N
(r_k^N-\varepsilon^N)}{N} \right)^{\frac{1}{p}}.
\end{array}
\end{equation}

Taking into account that
$\lambda_1\left(B_\varepsilon\right)=\lambda_1(B_1)/\varepsilon^2$,
$\lambda_1\left(B_{r_k}\right)=\mu_k$, using $N>2p$, and applying
Lemma \ref{autovalorescontinuos}, we conclude

$$\lim_{\varepsilon\to 0}\Vert a-\mu_k \Vert_{L^p(B_1)}\leq\lim_{\varepsilon\to 0}\left(\frac{\lambda_1(B_1)^p}{\varepsilon^{2p}}\frac{\omega_N \varepsilon^N}{N}+
\left(\lambda_1\left(A(\varepsilon,r_k)\right)-\mu_k\right)^p\frac{\omega_N
(r_k^N-\varepsilon^N)}{N}\right)^{\frac{1}{p}}=0,$$

\noindent and the proof is complete. \end{proof}

\begin{lemma}\label{critico2}
Let $k\geq 0$, $N=2$ and $p=1$. Then $\gamma_{1,k}=0$.
\end{lemma}

\begin{proof}
If $k=0$, this lemma follows from \cite[Lem. 3.2]{camovijfa}. In
this lemma a family of bounded, positive and radial solutions were
used. Hence, for the rest of the proof we will consider $k\geq 1$.

Similarly to the proof of the previous lemma, we will construct
some explicit sequences in $\Gamma_k$. In this case, this
construction will be slightly more complicated. First, for every
$\alpha\in (0,1)$, define $v_\alpha, A_\alpha :B_1\rightarrow
\real$ as the radial functions:

\begin{equation}  v_\alpha(r) = \left \{%
\begin{array}{l}
\alpha (1-r^2)(3-r^2)-\log r, \ \ \ \ \ \ \ \ \ \ \ \ \ \ \mbox{if} \ \ \ \alpha\leq r < 1 , \\ \\
\alpha (1-r^2)(3-r^2)-\log
\alpha+\displaystyle{\frac{\alpha^2-r^2}{2\alpha^2}}, \ \ \ \ \ \
\ \ \ \ \ \ \ \ \mbox{if} \ \ \ r < \alpha ,
\end{array}
\right.
\end{equation}

\begin{equation}  A_\alpha(r) = \left \{%
\begin{array}{l}
\displaystyle{\frac{16\alpha(1-r^2)}{\alpha (1-r^2)(3-r^2)-\log r}}, \ \ \ \ \ \ \ \ \ \ \ \ \ \ \mbox{if} \ \ \ \alpha< r < 1 , \\ \\
\displaystyle{\frac{16\alpha(1-r^2)+\frac{2}{\alpha^2}}{\alpha
(1-r^2)(3-r^2)-\log
\alpha+\displaystyle{\frac{\alpha^2-r^2}{2\alpha^2}}}}, \ \ \ \ \
\ \ \ \ \ \ \ \ \ \mbox{if} \ \ \ r < \alpha ,
\end{array}
\right.
\end{equation}

\noindent where $r=\vert x\vert$. It is easily seen that
$v_\alpha\in C^1(\overline{B_1})$, $A_\alpha\in L^\infty(B_1)$,
and

\begin{equation}
\left.
\begin{array}{cl}
\Delta v_\alpha(x)+A_\alpha(x)v_\alpha(x)= 0, & x\in B_1 \, \\
 v_\alpha(x)=0,\ \ \ \ \ \ \ &  x\in \partial B_1
\end{array}
\right\}
\end{equation}

Now, for every $\alpha\in(0,1)$ and $\varepsilon\in(0,r_k)$,
define $u_{\alpha,\varepsilon}:B_1\rightarrow \real$ as the radial
function:

\begin{equation}\label{solucionsubcritico2} u_{\alpha,\varepsilon}(x) = \left \{%
\begin{array}{l}
\varphi_k(x), \ \ \ \ \ \ \ \ \ \ \ \ \ \ \mbox{if} \ \ \ r_k\leq \vert x\vert < 1 , \\ \\
\phi_1\left( A(\varepsilon,r_k)\right)(x), \  \mbox{if} \ \ \ \varepsilon \leq \vert x\vert < r_k , \\ \\
\displaystyle{v_\alpha\left(\frac{x}{\varepsilon}\right)}, \ \ \ \
\ \ \ \ \mbox{if} \ \ \ \vert x\vert < \varepsilon .
\end{array}
\right.
\end{equation}

\noindent where the eigenfunctions $\varphi_k$ and $\phi_1\left(
A(\varepsilon,r_k)\right)$ are chosen such that
$u_{\alpha,\varepsilon}\in C^1(\overline{B_1})$.

An easy computation shows that $u_{\alpha,\varepsilon}$ is a
solution of (\ref{ceq6}), being $a_{\alpha,\varepsilon}\in
L^\infty(B_1)$ the radial function

\begin{equation}\label{bisinfimosubcritico} a_{\alpha,\varepsilon}(x) = \left \{%
\begin{array}{l}
\mu_k, \ \ \ \ \ \ \ \ \ \ \ \ \ \ \mbox{if} \ \ \ r_k< \vert x\vert < 1 , \\ \\
\lambda_1\left( A(\varepsilon,r_k)\right), \  \mbox{if} \ \ \ \varepsilon < \vert x\vert < r_k , \\ \\
\displaystyle{\frac{1}{\varepsilon^2}A_\alpha\left(\frac{x}{\varepsilon}\right)},
\ \ \ \ \ \ \ \ \mbox{if} \ \ \ \vert x\vert < \varepsilon .
\end{array}
\right.
\end{equation}

\noindent Again, using that the first Dirichlet eigenvalue
$\lambda_1(\Omega)$ is strictly decreasing with respect to the the
domain $\Omega$, it follows that

$$\lambda_1\left(A(\varepsilon,r_k)\right)>\lambda_1\left(B_{r_k}\right)=\mu_k.$$

Moreover, $\displaystyle{\inf_{\vert x\vert <\varepsilon}
a_{\alpha,\varepsilon}(x)=\left(\inf_{x\in B_1}
A_\alpha(x)\right)/\varepsilon^2:=m_\alpha/\varepsilon^2}$. We see
at once that $m_\alpha>0$ for every $\alpha\in(0,1)$. Hence, if we
fix $\alpha$ and choose $\varepsilon\in(0,1)$ such that
$\displaystyle{m_\alpha/\varepsilon^2\geq \mu_k}$, it is deduced
that $a_{\alpha,\varepsilon}\in\Gamma_k$.

Let us estimate the $L_1$-norm of $a_{\alpha,\varepsilon}-\mu_k$:

\begin{equation}\label{chorizo2} \begin{array}{l}
\displaystyle{\Vert a_{\alpha,\varepsilon}-\mu_k
\Vert_{L^1(B_1)}=\int_{B_\varepsilon}\left(\frac{1}{\varepsilon^2}A_\alpha\left(\frac{x}{\varepsilon}\right)
-\mu_k\right) dx+\int_{A(\varepsilon,r_k)}\left(\lambda_1 \left(
A(\varepsilon,r_k)\right)-\mu_k\right)dx}.
\end{array}
\end{equation}

Doing the change of variables $x=\varepsilon y$ in the first
integral and applying Lemma \ref{autovalorescontinuos} in the
second one, it is obtained, for fixed $\alpha\in (0,1)$:

$$\lim_{\varepsilon \to 0}\Vert a_{\alpha,\varepsilon}-\mu_k \Vert_{L^1(B_1)}=\int_{B_1}A_\alpha (y)dy.$$

Thus, from the definition of $\gamma_{1,k}$ we have

\begin{equation}\label{alfa}
\gamma_{1,k}\leq \int_{B_1}A_\alpha (y)dy\, ,\ \ \forall \alpha\in
(0,1).
\end{equation}
Now we will take limit when $\alpha$ tends to $0$ in this last
expression. For this purpose we first deduce easily from the
definition of $A_\alpha$ that $A_\alpha(r)\leq 16\alpha
(1-r^2)/(-\log r)\leq 32\alpha$ if $r\in (\alpha,1)$ and
$A_\alpha(r)\leq \left(16\alpha+2/\alpha^2\right)/(-\log \alpha)$
if $r\in (0,\alpha)$. It follows that

$$\int_{B_1}A_\alpha (y)dy=2\pi\int_0^1 r \, A_\alpha(r)dr\leq 2\pi\int_0^\alpha r\,\frac{16\alpha+2/\alpha^2}{-\log\alpha}dr+2\pi\int_\alpha^1 r \,32\alpha dr$$
$$=\pi\frac{16\alpha^3+2}{-\log\alpha}+32\pi\alpha(1-\alpha^2),$$

\noindent which gives $\lim_{\alpha\to 0}\int_{B_1}A_\alpha
(y)dy=0$ and the lemma follows from (\ref{alfa}). \end{proof}

\section{The supercritical case}

In this section, we study the supercritical case, i.e.
$p>\frac{N}{2}$, if $N\geq 2$. In all those cases we will prove
that $\gamma_{p,k}$ is strictly positive and that it is attained.
We begin by studying the case $p=\infty.$

\begin{lemma}\label{linfinito}
Let $k\geq 0$, $N\geq 2$ and $p=\infty$. Then
$\gamma_{\infty,k}=\mu_{k+1}-\mu_k$ is attained in the unique
element $a_0\equiv \mu_{k+1}\in\Gamma_k$.
\end{lemma}

\begin{proof}
Clearly $a_0\equiv\mu_{k+1}\in\Gamma_k$ satisfies $\Vert
a_0-\mu_k\Vert_{L^\infty(B_1)}=\mu_{k+1}-\mu_k$. Suppose, contrary
to our claim, that there exists $\mu_{k+1}\not\equiv a\in\Gamma_k$
such that $\Vert a-\mu_k\Vert_{L^\infty(B_1)}\leq\mu_{k+1}-\mu_k$.
Therefore $\mu_k \prec a\prec \mu_{k+1}$, a contradiction with the
fact $a\in\Gamma_k$ (see \cite{dolp}, \cite{vido}).
\end{proof}

Next we concentrate on the case $\frac{N}{2} < p < \infty.$

\begin{lemma}\label{cerosseparados}
Let $N\geq 2$, $p>N/2$ and $M>0$. Then, there exists
$\varepsilon=\varepsilon(N,p,M)$ with the following property:

For every $a\in L^p(B_1)$ satisfying $\Vert a\Vert_{L^p(B_1)}\leq
M$ and every $u\in H^1(B_1)$ radial nontrivial solution of
$-\Delta u=a\, u$ in $B_1$ we have
\begin{enumerate}
\item[i)] $z>\varepsilon$ for every zero $z$ of $u$. \item[ii)]
$\vert z_2-z_1\vert>\varepsilon$ for every different zeros
$z_1,z_2$ of $u$.
\end{enumerate}
\end{lemma}

\begin{proof}
Let $z\in (0,1]$ be a zero of $u$. Hence, multiplying the equation
$-\Delta u=a\, u$ by $u$, integrating by parts in the ball $B_z$
and applying H\"{o}lder inequality, we obtain

$$\int_{B_z}\vert\nabla u\vert^2=\int_{B_z}a\, u^2\leq \Vert a\Vert_{L^p(B_z)}\Vert u\Vert^2_{L^\frac{2p}{p-1}(B_z)}.$$

From the above it follows that

$$M\geq\Vert a\Vert_{L^p(B_1)}\geq\Vert a\Vert_{L^p(B_z)}\geq\frac{\Vert\nabla u\Vert^2_{L^2(B_z)}}{\Vert u\Vert^2_{L^\frac{2p}{p-1}(B_z)}}\geq
\min_{v\in H^1_0(B_z)}\frac{\Vert\nabla v\Vert^2_{L^2(B_z)}}{\Vert
v\Vert^2_{L^\frac{2p}{p-1}(B_z)}}.$$

From the change $w(x)=v(z\, x)$, it is easily deduced that

$$\min_{v\in H^1_0(B_z)}\frac{\Vert\nabla v\Vert^2_{L^2(B_z)}}{\Vert v\Vert^2_{L^\frac{2p}{p-1}(B_z)}}=z^{\frac{N}{p}-2}\min_{w\in H^1_0(B_1)}\frac{\Vert\nabla w\Vert^2_{L^2(B_1)}}{\Vert w\Vert^2_{L^\frac{2p}{p-1}(B_1)}}:=z^{\frac{N}{p}-2}\alpha(N,p),$$

\noindent where we have used the compact embedding
$\displaystyle{H_0^1(B_1)\subset L^\frac{2p}{p-1}(B_1)}$ (since
$p>N/2$, then $2<\frac{2p}{p-1}<\frac{2N}{N-2}$, which is the
critical Sobolev exponent). Thus, taking $\varepsilon_1>0$ such
that $M<{\varepsilon_1}^{\frac{N}{p}-2}\alpha(N,p)$, we conclude
part i) of the lemma with $\varepsilon=\varepsilon_1$.

For the second part of the lemma, consider two zeros $0<z_1<z_2<1$
of $u$. Taking into account that $z_1\geq\varepsilon_1$ and
arguing in the same manner of part i), we obtain

$$M\geq\Vert a\Vert_{L^p(B_1)}\geq\Vert a\Vert_{L^p\left(A\left( z_1,z_2\right)\right)}\geq\frac{\Vert\nabla u\Vert^2_{L^2\left(A\left( z_1,z_2\right)\right)}}{\Vert u\Vert^2_{L^\frac{2p}{p-1}\left(A\left( z_1,z_2\right)\right)}}=$$

$$\displaystyle{\frac{\omega_N\int_{z_1}^{z_2}r^{N-1}u'(r)^2dr}{\left(\int_{z_1}^{z_2}\omega_N r^{N-1}\vert u(r)\vert^{2p/(p-1)}dr\right)^{(p-1)/p}}\geq \omega_N^{1/p}\varepsilon_1^{N-1}\frac{\Vert u' \Vert^2_{L^2(z_1,z_2)}}{\Vert u\Vert^2_{L^\frac{2p}{p-1}(z_1,z_2)}}}.$$

On the other hand, from the one dimensional change of variable
$w(x)=v(z_1+(z_2-z_1)x)$, it is immediate that

$$\min_{v\in H^1_0(z_1,z_2)}\frac{\Vert v' \Vert^2_{L^2(z_1,z_2)}}{\Vert v\Vert^2_{L^\frac{2p}{p-1}(z_1,z_2)}}=(z_2-z_1)^{\frac{1}{p}-2}\min_{w\in H^1_0(0,1)}\frac{\Vert w' \Vert^2_{L^2(0,1)}}{\Vert w\Vert^2_{L^\frac{2p}{p-1}(B_1)}}:=(z_2-z_1)^{\frac{1}{p}-2}C_p.$$

It follows that
$M\geq\omega_N^{1/p}\varepsilon_1^{N-1}(z_2-z_1)^{1/p-2}C_p$. From
this, taking $\varepsilon_2$ such that
$M<\omega_N^{1/p}\varepsilon_1^{N-1}\varepsilon_2^{1/p-2}C_p$, we
conclude part ii) of the lemma with $\varepsilon=\varepsilon_2$.

Obviously, taking $\varepsilon=\min\{ \varepsilon_1,\varepsilon_2
\}$, the lemma is proved.

\end{proof}

\begin{lemma}\label{sealcanza}
Let $k\geq 0$, $N\geq 2$ and $N/2<p<\infty$. Then $\gamma_{p,k}$
is strictly positive and it is attained in a function
$a_0\in\Gamma_k$.
\end{lemma}

\begin{proof}

Take a sequence $\{a_n\}\subset\Gamma_k$ such that $\Vert
a_n-\mu_k\Vert_{L^p(B_1)}\to\gamma_{p,k}$. Take $\{u_n\}\subset
H^1(B_1)$ such that $u_n$ is a radial solution of (\ref{ceq6}),
for $a=a_n$, with the normalization $\Vert
u_n\Vert_{H^1(B_1)}^2=\int_{B_1}\left(\vert\nabla
u_n\vert^2+u_n^2\right)=1$. Therefore, we can suppose, up to a
subsequence, that $u_{n} \rightharpoonup u_{0}$ in $H^1(B_1)$ and
$u_{n} \rightarrow u_{0}$ in $L^\frac{2p}{p-1}(B_1)$ (since
$p>N/2$, then $2<\frac{2p}{p-1}<\frac{2N}{N-2}$, which is the
critical Sobolev exponent). On the other hand, since $\{a_n\}$ is
bounded in $L^p(B_1)$, and $1\leq N/2<p<\infty$, we can assume, up
to a subsequence, that $a_{n} \rightharpoonup a_{0}$ in
$L^p(B_1)$. Taking limits in the equation (\ref{ceq6}), for
$a=a_n$ and $u=u_n$, we obtain that $u_0$ is a solution of this
equation for $a=a_0$. Note that $u_{n} \rightarrow u_{0}$ in
$L^\frac{2p}{p-1}(B_1)$ and $a_{n} \rightharpoonup a_{0}$ in
$L^p(B_1)$ yields $\lim\int_{B_1}\vert\nabla
u_n\vert^2=\lim\int_{B_1}a_n u_n^2=\int_{B_1}a_0
u_0^2=\int_{B_1}\vert\nabla u_0\vert^2$ and consequently
$u_n\rightarrow u_0\not\equiv 0$ in $H^1(B_1)$. Therefore, if
$a_0\not\equiv \mu_k$, then $a_0\in\Gamma_k$ and $\Vert
a_0-\mu_k\Vert_p\leq\lim_{n\to\infty}\Vert
a_n-\mu_k\Vert_p=\gamma_{p,k}$, and the lemma follows.

On the contrary, suppose by contradiction that $a_0\equiv\mu_k$.
Then $u_0=\varphi_k$ for some nontrivial radial eigenfunction
$\varphi_k$. Consider $\varepsilon$ given in Lemma
\ref{cerosseparados}. Take $\varepsilon_0=\min\left\{ \varepsilon,
2r_k/3, 2(1-r_1), r_i-r_{i+1};\ 1\leq i\leq k-1\right\}$. Thus,
from the previous lemma, $u_n$ has no zeros in
$(0,\varepsilon_0)$, and has, at most, one zero in each of the $k$
disjoint intervals $(r_i-\varepsilon_0/2,r_i+\varepsilon_0/2)$,
$1\leq i\leq k$. Therefore, $u_n$ has, at most, $k$ zeros in the
set $A:=(0,\varepsilon_0)\bigcup\left ( \cup_{1\leq i\leq k}\
(r_i-\varepsilon_0/2,r_i+\varepsilon_0/2)\right )$.

On the other hand, taking into account the continuous embedding
\linebreak  $H_{rad}^1\left(A(\varepsilon_0,1)\right)\subset
C\left(A(\varepsilon_0,1)\right)$ and $u_n\rightarrow\varphi_k$ in
$H_0^1(B_1)$, we can assert $u_n\rightarrow\varphi_k$ in
$C\left(A(\varepsilon_0,1)\right)$. Clearly $\min_{r\in
(0,1]\setminus A}\vert \varphi_k(r)\vert>0$. Then, for large $n$
we see that $\min_{r\in (0,1]\setminus A}\vert u_n(r)\vert>0$,
which implies that $u_n$ does not vanish in $(0,1]\setminus A$,
for large $n$. Since $u_n$ has, at most, $k$ zeros in $A$, we
conclude that $u_n$ has, at most, $k$ zeros in (0,1], for large
$n$. This contradicts Proposition \ref{n+1ceros} and the lemma
follows.

\end{proof}

\section{The critical case}

In this section, we study the critical case, i.e. $p=\frac{N}{2}$,
if $N\geq 3$. We will prove that $\gamma_{p,k}>0$.

\begin{lemma}\label{N/2} Let $k\geq 0$, $N\geq 3$ and $p=N/2$. Then $\gamma_{p,k}>0.$

\end{lemma}

\begin{proof}
To obtain a contradiction, suppose that $\gamma_{p,k}=0$. Then we
could find a sequence $\{a_n\}\subset\Gamma_k$ such that
$a_n\rightarrow\mu_k$ in $L^{N/2}(B_1)$. Similarly to the
supercritical case, we can take $\{u_n\}\subset H^1(B_1)$ such
that $u_n$ is a radial solution of (\ref{ceq6}), for $a=a_n$, with
the normalization $\Vert u_n\Vert_{H^1(B_1)}^2=1$. Again, we can
suppose, up to a subsequence, that $u_{n} \rightharpoonup u_{0}$
in $H^1(B_1)$ and taking limits in the equation (\ref{ceq6}), for
$a=a_n$ and $u=u_n$, we obtain that $u_0$ is a solution of this
equation for $a=\mu_k$.

We claim that $u_n\rightarrow u_0$ in $H^1(B_1)$ and consequently,
$u_0=\varphi_k$, for some nontrivial eigenfunction $\varphi_k$.
For this purpose, we set

$$\lim\int_{B_1}\vert\nabla u_n\vert^2=\lim\int_{B_1}a_n u_n^2=\lim\int_{B_1}(a_n-\mu_k)u_n^2+\lim\int_{B_1}\mu_ku_n^2=$$

$$0+\mu_k\int_{B_1}u_0^2=\int_{B_1}\vert\nabla u_0\vert^2,$$

\noindent where we have used $a_n\rightarrow\mu_k$ in
$L^{N/2}(B_1)$ and $u_n^2$ is bounded in $L^{N/(N-2)}(B_1)$ (since
$u_n$ is bounded in $H^1(B_1)\subset L^{2N/(N-2)}(B_1)$). Thus,
from standard arguments, we deduce that $u_n\rightarrow
u_0=\varphi_k$ in $H^1(B_1)$.

In the following, we will fix $\varepsilon\in(0,r_k)$. Since
$a_n\rightarrow\mu_k$ in $L^{N/2}\left(A(\varepsilon,1)\right)$
and $u_n\rightarrow u_0=\varphi_k$ in
$H_{rad}^1\left(A(\varepsilon,1)\right)\subset
C\left(A(\varepsilon,1)\right)$, we can assert that
$a_nu_n\rightarrow\mu_k\varphi_k$ in
$L^{N/2}\left(A(\varepsilon,1)\right)\subset
L^1\left(A(\varepsilon,1)\right)$. Thus $-\Delta
u_n\rightarrow\mu_k\varphi_k$ in
$L^1\left(A(\varepsilon,1)\right)$, which yields
$u_n\rightarrow\varphi_k$ in $C^1\left(A(\varepsilon,1)\right)$.
It follows that, for large $n$, the number of zeros of $u_n$ is
equal to the number of zeros of $\varphi_k$ in the annulus
$A(\varepsilon,1)$, which is exactly $k$. Applying Proposition
\ref{n+1ceros} we can assert that, for large $n$ there exists a
zero $\varepsilon_n\in (0,\varepsilon]$ of $u_n$. Hence,
multiplying the equation $-\Delta u_n=a_nu_n$ by $u_n$,
integrating by parts in the ball $B_{\varepsilon_n}$ and applying
H\"{o}lder inequality, we deduce

$$\int_{B_{\varepsilon_n}}\vert\nabla u_n\vert^2=\int_{B_{\varepsilon_n}}a_n u_n^2\leq \Vert a_n\Vert_{L^{N/2}(B_{\varepsilon_n})}\Vert u_n\Vert^2_{L^{2N/(N-2)}(B_{\varepsilon_n})}.$$

From the above it follows that

$$\Vert a_n\Vert_{L^{N/2}(B_{\varepsilon_n})}\geq\frac{\Vert\nabla u_n\Vert^2_{L^2(B_{\varepsilon_n})}}{\Vert u_n\Vert^2_{L^{2N/(N-2)}(B_{\varepsilon_n})}}\geq
\inf_{u\in H^1_0(B_{\varepsilon_n})}\frac{\Vert\nabla
u\Vert^2_{L^2(B_{\varepsilon_n})}}{\Vert
u\Vert^2_{L^{2N/(N-2)}(B_{\varepsilon_n})}}.$$

From the change $v(x)=u(\varepsilon_n x)$, it is easily deduced
that

$$\inf_{u\in H^1_0(B_{\varepsilon_n})}\frac{\Vert\nabla u\Vert^2_{L^2(B_{\varepsilon_n})}}{\Vert u\Vert^2_{L^{2N/(N-2)}(B_{\varepsilon_n})}}=\inf_{v\in H^1_0(B_1)}\frac{\Vert\nabla v\Vert^2_{L^2(B_1)}}{\Vert v\Vert^2_{L^{2N/(N-2)}(B_1)}}:=C_N>0.$$

From the above it follows that, for fixed $\varepsilon\in (0,r_k)$
and large $n$, we obtain

$$C_N\leq\Vert a_n\Vert_{L^{N/2}(B_{\varepsilon_n})}\leq \Vert a_n-\mu_k\Vert_{L^{N/2}(B_{\varepsilon_n})}+\Vert \mu_k\Vert_{L^{N/2}(B_{\varepsilon_n})}\leq$$

$$\Vert a_n-\mu_k\Vert_{L^{N/2}(B_1)}+\Vert \mu_k\Vert_{L^{N/2}(B_{\varepsilon})}.$$

Taking limits when $n$ tends to $\infty$ in this expression we
deduce

$$C_N\leq \mu_k\left(\frac{\omega_N\varepsilon^N}{N}\right)^{2/N}.$$

Choosing $\varepsilon>0$ sufficiently small we obtain a
contradiction.

\end{proof}


\end{document}